\documentclass{amsart}

\usepackage{graphicx}
\usepackage{amssymb}
\usepackage{amsthm}
\usepackage{color}
\usepackage{colortbl}
\definecolor{grey}{rgb}{0.75,0.75,0.75}
\newtheorem{theorem}{Theorem}[section]
\newtheorem{corollary}[theorem]{Corollary}

\newtheorem{proposition}[theorem]{Proposition}

\theoremstyle{definition}

\theoremstyle{remark}
\newtheorem{remark}[theorem]{Remark}

\newcommand{\Pn}{\mathbb{P}^2}
\newcommand{\Rn}{\mathbb{R}^2}
\newcommand{\defn}[1]{{\em #1}}
\newcommand{\drawat}[3]{\makebox[0pt][l]{\raisebox{#2}{\hspace*{#1}#3}}}

\begin{document}

\title[Triangles in arrangements of lines and pseudo-lines]{On simple arrangements of lines and pseudo-lines in $\Pn$ and $\Rn$ with the maximum number of triangles}

\author{Nicolas Bartholdi}
\address{N. Bartholdi, Universit\' e de Gen\`eve,
Section de math\'ematiques, 
2-4 rue du Li\`evre,
CP 64, 1211 Gen\`eve 4 (Switzerland)}

\author{J\'er\'emy Blanc}
\address{J. Blanc, Laboratoire J.A. Dieudonn\'e (UMR 6621),
Universit\'e de Nice Sophia Antipolis - C.N.R.S., 
Facult\'e des Sciences - Parc Valrose,
06108 Nice cedex 2 (France)}
\author{S\'ebastien Loisel}
\address{S. Loisel, Department of Mathematics,
Wachman Hall,
1805 North Broad Street,
Temple University, 
Philadelphia, PA 19122
(USA)}



\maketitle
\section{Introduction}
 \label{Sec:Intro}
 A curve $\Gamma\subset \Pn=\Pn(\mathbb{R})$ (respectively $\subset \Rn$) is a \defn{projective} (respectively an \defn{affine}) \defn{pseudo-line} if there is a homeomorphism $\phi:\Pn \rightarrow \Pn$ (respectively $\phi:\Rn \rightarrow \Rn$) such that $\phi(\Gamma)$ is a line.
 
 A \defn{projective} (respectively an \defn{affine}) \defn{arrangement of $n$ pseudo-lines} is a set of $n$ pseudo-lines in $\Pn$ (respectively in $\Rn$), such that any pair of pseudo-lines intersects in exactly one point. A \defn{projective} (respectively an \defn{affine}) \defn{arrangement of lines} is such an arrangement where each pseudo-line is a line. Note that there usually is no homeomorphism $\phi$ of the plane turning a pseudo-line arrangement $\mathcal{A}$ into a line arrangement $\phi(\mathcal{A})$ (\cite{bib:Gru}, page 42, Theorem 3.2). Our sole interest is with \defn{simple arrangements}, i.e. arrangements without multiple intersections.

A simple projective arrangement $\mathcal{A}$ of $n$ lines decomposes the projective plane $\Pn$ into ${n(n-1)}/2+1$ polygons; we will denote by $p_3(\mathcal{A})$ the number of \defn{triangles} obtained. We can do the same for pseudo-lines, as a triangle is a region delimited by exactly three pseudo-lines of the arrangement. Similarly, in the Euclidean plane $\Rn$, we denote by $a_3(\mathcal{A'})$ the number of (bounded) triangles delimited by an affine arrangement $\mathcal{A'}$.

It was originally proposed by Gr\"unbaum \cite{bib:Gru} to look for arrangements with many triangles, and there is already a substantial literature on this question.

Denote by $p_3^s(n)$ (respectively $a_3^s(n)$) the maximal number of triangles that can be obtained with a simple arrangement of $n$ lines in the projective plane (respectively in the Euclidean plane).
We denote by $\overline{p_3^s}(n)$ and $\overline{a_3^s}(n)$ the same notions for pseudo-lines. A projective arrangement $\mathcal{A}$ of $n$ pseudo-lines such that $p_3(\mathcal{A})=\overline{p_3}(n)$ is classically called \emph{$p_3$-maximal}. Here we will only say that the arrangement is \emph{maximal}, and will use the same terminology for affine arrangements.

Then, an easy observation on the number of segments shows that if $n\geq 4$, the following relations occur:
\begin{equation}\label{stupidbound}
\begin{array}{cccccc}
p_3^s(n)&\leq& \overline{p_3^s}(n)&\leq& n(n-1)/3\\
\rotatebox[origin=c]{270}{\ensuremath \geq} & & \rotatebox[origin=c]{270}{\ensuremath \geq} & & \rotatebox[origin=c]{270}{\ensuremath \geq}\\
a_3^s(n)&\leq& \overline{a_3^s}(n)&\leq & n(n-2)/3.\end{array}\end{equation}
We will say that a segment is \emph{used} if it is a part of one triangle of the arrangement, and say that it is \emph{unused} otherwise. 
An arrangement satisfying the equality with the bound above is an arrangement whose segments are all used in one triangle -- we will say in this case that it is a \emph{perfect arrangement}. Note that a perfect arrangement is maximal, but the converse is false in general.

There is currently no known $n$ such that  $\overline{p_3^s}(n)\gneqq p_3^s(n)$ or $\overline{a_3^s}(n)\gneqq a_3^s(n)$.

Infinitely many examples of integers $n \equiv 0,4\pmod{6}$ are known to satisfy $\overline{p_3^s}(n)=n(n-1)/3$ (see \cite{bib:Ha1}, \cite{bib:Ha2}, \cite{bib:Rou1}), an algorithm to find these was given in \cite{bib:BoRoSt}, and the only counterexample previously known is $n=12$ (see \cite{bib:Rou2}). A construction in straight lines has been given in \cite{bib:FoR} to prove that  $p_3^s(n)=n(n-1)/3$ for $n=2 \cdot 2^t+2$, for any integer $t\geq 0$ -- we generalise this for more infinite sequences in Theorem \ref{TheoNewSequence}.

The projective examples of pseudo-lines lead to similar affine configurations by putting one of the pseudo-lines at infinity and by removing it. In particular, $\overline{p_3^s}(n)=n(n-1)/3$ if and only if $\overline{a_3^s}(n-1)=(n-1)(n-3)/3$, and the same if true for arrangements of straight lines (i.e. for $p_3^s$ and $a_3^s$). There exist thus infinitely many examples of integers $n\equiv 3,5 \pmod{6}$ such that $\overline{a_3^s}(n)=n(n-2)/3$, and we have also  $a_3^s(n)=n(n-2)/3$ for $n=2 \cdot 2^t+1$.

The projective odd case is worse than the even case: it was observed by J. Granham (\cite{bib:Gru}, page 26, Theorem 2.21) that \begin{equation}\label{eq:Granham}\overline{p_3^s}(n)\leq n(n-2)/3\hspace{0.5 cm}\mbox{\it if $n>3$ is odd}.\end{equation} 
Conversely, the \emph{affine} even case is worse than the odd case. 
We give a new bound in this case:
\begin{theorem}\label{Thm:Affinebound}
If $n$ is an even integer, then $\overline{a_3^s}(n)\leq \lfloor n(n-7/3)/3\rfloor$.\end{theorem}
The bound of Theorem \ref{Thm:Affinebound} is reached for  $4,6,10,16$ pseudo-lines but not for $8,12,14$ pseudo-lines (see Theorem \ref{thm:resultssmallvalues}).
Note that adding one line to an affine perfect arrangement of $n-1\equiv 3,5\pmod{6}$ lines, we obtain infinitely many examples of values of $n\equiv 0,4 \pmod{6}$ lines where $\overline{a_3^s}(n)\geq n(n-5/2)/3$, which is close to the polynomial of Theorem \ref{Thm:Affinebound}. It would be interesting to find the best polynomial upper bound, which is between the two above. Note that there exists no even integer $n$ where $\overline{a_3^s}(n)>n(n-5/2)/3$ has been proved.

\bigskip

Remark that the bounds of (\ref{stupidbound}) are sometimes not integers, and thus may not be attained, even if the parity is good. In the affine odd case, taking the integer part of (\ref{stupidbound}) is a tight bound, as we will provide infinitely many examples of maximal arrangements of $n\equiv 1 \pmod{6}$ lines with $\lfloor n(n-2)/3 \rfloor$ triangles (Theorem \ref{TheoNewSequence}). However, in the projective even case, the bound may be improved, to give the following result, which seems to be already known (see \cite{bib:Rou3} Table I), but we were not able to find a proof in the literature.
\begin{proposition}\label{proptwomodsix}
If $n \equiv 2\pmod{6}$, then $\overline{p_3^s}(n)\leq \lfloor n(n-1)/3 \rfloor-1$.
\end{proposition}
Note that this bound could be improved, as there is still no known example where the equality occurs.

\bigskip

We give then in Proposition \ref{Prp:duplicationG} a way to obtain a good affine arrangement of $2n-1$ lines starting from another good one of $n$ lines. The construction is homeomorphic to those of \cite{bib:Ha1} and \cite{bib:FuP} but has two advantages: it is explicit and may be used starting from non-perfect arrangements. This gives in particular the following new sequences:

\begin{theorem}\label{TheoNewSequence}For any integer $t \geq 0$ we have
\begin{center}$\begin{array}{lllll}
\mbox{if }n=14 \cdot 2^t +1,& a_3^s(n)&=&n(n-2)/3;\\
\mbox{if }n=6 \cdot 2^t+1,&a_3^s(n)&=& \lfloor n(n-2)/3 \rfloor=(n(n-2)-2)/3;\\
\mbox{if }n=18 \cdot 2^t+1,&\overline{a_3^s}(n)&=& \lfloor n(n-2)/3 \rfloor=(n(n-2)-2)/3.\end{array}$\end{center}
\end{theorem}
In particular, this shows that the best polynomial upper bound for $a_3^s(n)$, $\overline{a_3^s}(n)$, $p_3^s(n)$ and $\overline{p_3^s}(n)$, $n\equiv 1 \pmod{6}$ is $(n(n-2)-2)/3=\lfloor n(n-2)/3\rfloor$.

\bigskip

Finally, we are able to describe the explicit values of $a_3^s(n)$, $\overline{a_3^s}(n)$, $p_3^s(n)$, $\overline{p_3^s}(n)$ for small values of $n$. A computer program -- described in Section \ref{Sec:DescAlgo} -- allows us to find explicitly some values of $\overline{a_3^s}(n)$. Using the bounds and sequences described above, and the relation between $\overline{a_3^s}(n)$ and $\overline{p_3^s}(n+1)$ -- adding the line at infinity -- we obtain the following result:

\begin{theorem}
\label{thm:resultssmallvalues}
The values of $a_3^s(n)$, $\overline{a_3^s}(n)$, $p_3^s(n)$, $\overline{p_3^s}(n)$ are given in the table
\begin{center}\begin{flushleft}$\begin{array}{|c|c|c|c|c|c|c|c|c|c|c|c|c|c|c|c|c|c|c|c|c|c|c|c|c|c|c|c|}
\hline
n & 3 & 4 & 5 & 6 & 7 & 8 & 9 & 10 & 11 & 12 & 13 & 14 & 15 & 16 \\
\hline
\overline{p_3^s}(n)&\cellcolor{grey}{ \bf{4} }&\cellcolor{grey}{ \bf{4} }& \cellcolor{grey}{ \bf{5} } &\cellcolor{grey}{ \bf{10} }& \cellcolor{grey}{ \bf{11} }& \cellcolor{grey}{\underline{{\bf 16}}} &  \cellcolor{grey}{\bf{21}} &  \cellcolor{grey}{\bf{30}} & \cellcolor{grey}{\underline{\bf 32}-\it{33}} & \cellcolor{grey}{\underline{42}} & {\bf 47} & \underline{\bf 58}-\colorbox{grey}{\it{59}} & \cellcolor{grey}{\bf{65}} & \cellcolor{grey}{\bf{80}}  \\
\hline
\overline{a_3^s}(n)&\cellcolor{grey}{\bf{1}} & \cellcolor{grey}{\bf{2}} & \cellcolor{grey}{\bf{5}} & {\bf{7}} & \cellcolor{grey}{\bf{11}} &  \underline{\bf 14} & \cellcolor{grey}{\bf{21}} & \bf{25} &  \underline{\bf 32} & \underline{\bf 37} & {\bf 47} & \underline{{\bf 53}} & \cellcolor{grey}{\bf{65}} & \bf{72} \\
\hline
\end{array}$

$\begin{array}{|c|c|c|c|c|c|c|c|c|c|c|c|c|c|c|c|c|c|c|c|c|c|c|c|c|c|c|}
\hline
n  & 17 & 18 & 19 &  20 & 21 & 22 & 23   \\
\hline
\overline{p_3^s}(n) &\cellcolor{grey}{\bf{85}}  & \cellcolor{grey}{\bf{102}} &  {107}  &  \underline{124}-\colorbox{grey}{{\it{125}}} & \cellcolor{grey}{{133}} & \cellcolor{grey}{154}  & \cellcolor{grey}{161} \\
\hline
\overline{a_3^s}(n) &  \cellcolor{grey}{\bf{85}}&  \underline{{\bf 93}}-\textit{94} & {107} & \underline{116}-{\it 117} & \cellcolor{grey}{133}  & \underline{143}-{\it 144}  & \cellcolor{grey}{161}\\
\hline
\end{array}$

$\begin{array}{|c|c|c|c|c|c|c|c|c|c|c|c|c|c|c|c|c|c|c|c|c|c|c|c|c|c|c|}
\hline
n  &24 & 25 & 26 & 27  & 28 & 29 & 30  \\
\hline
\overline{p_3^s}(n) & \cellcolor{grey}{184} & {\bf 191}  &  \underline{\bf 214}-\colorbox{grey}{{\textit{215}}} &\cellcolor{grey}{225}    & \cellcolor{grey}{252} & {\bf 261} &{\bf 290}  \\
\hline
\overline{a_3^s}(n)    & \underline{172}-{\it 173} & {\bf 191} & \underline{\bf 203}-\textit{205}  &\cellcolor{grey}{225}    & \underline{238}-{\it 239}  &{\bf 261}  & \underline{\bf 275}-{\it 276}\\
\hline
\end{array}$.\end{flushleft}\end{center}

An entry of the form $\begin{array}{|c|}\hline {\mathrm{x}}-\it{y} \\ \hline\end{array}$ means that we have an arrangement with $\mathrm{x}$ pseudo-lines but that the best upper bound is $y$. {\bf Bold} entries are known to be stretchable. \underline{Underlined} quantities are strictly smaller than the bounds given above. \begin{tabular}{c}\cellcolor{grey}Grey\end{tabular} entries were previously known (in particular in \cite{bib:Rou3} for pseudo-lines); we include them for completeness.
\end{theorem}
\section{Some new bounds -- Proof of Theorem \ref{Thm:Affinebound} and Proposition \ref{proptwomodsix}}
We prove Theorem \ref{Thm:Affinebound}, i.e. that for any even integer $n$, the inequality $\overline{a_3^s}(n)\leq \lfloor n(n-7/3)/3\rfloor$ holds.
\begin{proof}[Proof of Theorem \ref{Thm:Affinebound}]
Let $\mathcal{A}$ be an affine simple arrangement of $n$ pseudo-lines, with $n\geq 2$ an even number.\\ 
Suppose that a pseudo-line $L \in \mathcal{A}$  contains $n-2$ used segments, i.e. $L$ touches exactly $n-2$ triangles of the arrangement. We denote these triangles by $t_1,...,t_{n-2}$, such that $t_i$ and $t_{i+1}$ have a common vertex for $i=1,...,n-3$, denote by $M$ and $N$ the two pseudo-lines intersecting $L$ in the extremities (such that $M$ touches $t_1$ and $N$ touches $t_{n-2}$) and denote by $\Delta$ the region delimited by the three lines $L$, $M$,$N$ (which is not a "triangle" of our arrangement as other pseudo-lines intersect it). The $n-2$ triangles  touching $L$ are alternatively inside $\Delta$ and outside it. So, either $t_1$ or $t_{n-2}$ is not contained in $\Delta$. Without loss of generality, we assume that $t_{n-2}$ is not contained in $\Delta$ and illustrate the situation in Figure \ref{FigD}.
\begin{figure}[ht]{\includegraphics[width=10cm]{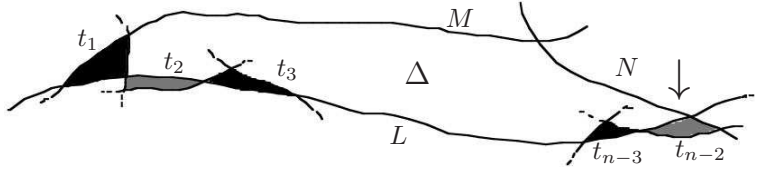}
\drawat{-42.5mm}{18.5mm}{$M$}%
\drawat{-20mm}{12mm}{$N$}%
\drawat{-50mm}{3mm}{$L$}%
\drawat{-91.5mm}{16mm}{$t_1$}%
\drawat{-48mm}{11mm}{\Large $\Delta$}%
\drawat{-80mm}{13mm}{$t_2$}%
\drawat{-65mm}{12mm}{$t_3$}%
\drawat{-23mm}{1mm}{$t_{n-3}$}%
\drawat{-12mm}{1.3mm}{$t_{n-2}$}%
\drawat{-13mm}{10mm}{\huge $\downarrow$}%
} \caption{The situation of the pseudo-line $L$\label{FigD}}\end{figure}
Note that the segment of the line $N$ which starts from $L$ and which is not contained in $t_{n-2}$ is not used. (On the figure, the segment with an arrow).

Then, to every pseudo-line that contains $n-2$ used segments, we associate the unused segment defined above. It does not belong to $L$, but it has one of its extremities on it. 
As the arrangement is simple, a segment cannot be associated to more than two pseudo-lines.

Denote by $m$ the number of pseudo-lines that contain exactly $n-2$ used segments; we associate to them at least $m/2$ unused segments.
Suppose, {\it ab absurdo}, that there are more than ${n(n-7/3)}/{3}$ triangles. Then at least $n(n-7/3)$ segments are used, so $n(n-7/3) + m/2  \leq n(n-2)$, which implies that $m \leq \frac{2}{3} n$.
But then, the number of used segments is at most 
$$m \cdot (n-2) + (n-m) \cdot (n-3) = n \cdot (n-3) + m \leq n \cdot (n-7/3),$$ which is a contradiction.
\end{proof}

\bigskip

We prove now Proposition \ref{proptwomodsix}, i.e. that $\overline{p_3^s}(n)<\lfloor \frac{n(n-1)}{3} \rfloor$ for any positive integer  $n \equiv 2\pmod{6}$.
\begin{proof}[Proof of Proposition \ref{proptwomodsix}]
Suppose that there exists some projective arrangement $\mathcal{A}$ of $n$ pseudo-lines with exactly $\lfloor \frac{n(n-1)}{3} \rfloor= \frac{n(n-1)-2}{3}$ triangles. Since the number of segments is not divisible by $3$, there exists at least one of them which is not touching any triangle. We choose then one pseudo-line that touches at most $n-2$ triangles of the arrangement; we stretch it to a line and remove it to get an affine arrangement of $n-1$ pseudo-lines, which has at least $\frac{n(n-1)-2}{3}-(n-2)$ triangles. But this number is strictly bigger than $\frac{(n-1)(n-3)}{3}$, which is not possible.
\end{proof}

\section{A way to construct maximal arrangements}

\begin{proposition}
\label{Prp:duplicationG}
Let $n\geq 2$ be an even number and let $\mathcal{A}=\{Y_0,L_1,...,L_n\}$ be a simple affine arrangement of $n+1$ lines, given by the equations 
\begin{center} $\begin{array}{rcl}
Y_0&:=&\{(x,y) \in \Rn \ | \ y=0\},\\
L_i&:=&\{(x,y) \in \Rn \ | \ y=m_i (x-a_i)\},\end{array}$\end{center}
 where
	\begin{center}
	$\begin{array}{l}
	\{a_1,...,a_{n-2}\}=\big\{\tan(\alpha)\ \big| \pm \alpha\in \{\frac{\pi}{n},2\frac{\pi}{n},...,\frac{\pi}{2}-\frac{\pi}{n}\}\big\}\\
	-\frac{1}{n} < a_{n-1} < 0 < a_{n} < \frac{1}{n};
	\end{array}$
	\end{center}
	and such that the line $Y_0$ touches exactly $n-1$ triangles (which means that every one of its segments is used in one triangle) of the affine arrangement $\mathcal{A}$.

Then, there exist $n$ lines $M_1,M_2,...,M_{n}$  given by the equations 
\begin{center} $M_i:=\{(x,y) \in \Rn \ | \ y=\mu_i (x-b_i)\}$,\end{center}
where $b_i=\tan(\beta_i)$ and 

\begin{center}
	$\begin{array}{l}
	\{\beta_1,...,\beta_{n}\}=\big\{-\frac{\pi}{2}+\frac{1}{2}\cdot\frac{\pi}{n},-\frac{\pi}{2}+\frac{3}{2}\cdot\frac{\pi}{n},...,-\frac{1}{2}\cdot\frac{\pi}{n},\frac{1}{2}\cdot\frac{\pi}{n},...,\frac{\pi}{2}-\frac{1}{2}\cdot\frac{\pi}{n}\big\}
	\end{array}$
	\end{center}
	 and such that 
the affine arrangement $\mathcal{B}=\{Y_0,L_1,\dots,L_{n},M_{1},...,M_{n}\}$ of $2n+1$ lines is simple and
has exactly $n^2$ triangles more than $\mathcal{A}$;
the line $Y_0$ touches exactly $2n-1$ triangles of the arrangement 
$\mathcal{B}$.

Explicitly, we can take 
\begin{center}$
\mu_i:=\sigma\cdot \frac{m_{\min}}{n^{10}}\cdot \big(\sin\big(2 \beta_i\big)+\frac{1}{n^6\cdot b_i}\big),
$\end{center}
where $\sigma:=1$ if $L_n$ and $L_{n-1}$ intersect in the upper half-plane and $\sigma:=-1$ otherwise, and where $m_{\min}:=\min\{|m_i| \ | \ i=1,..,n\}$.
\end{proposition}
\begin{remark}
If $|a_{n-1}|$ and $|a_n|$ are smaller than $1/2n$, then the new arrangement $\mathcal{B}$ also satisfies the conditions of the Proposition; this allows us to iterate the process if $|a_{n-1}|$ and $|a_n|$ are arbitrary small.\end{remark}
\begin{proof}
Write $\epsilon_1:=\frac{1}{n^{10}}$, $\epsilon_2:=\frac{1}{n^6}$. Multiplying all the slopes by $-1$ if needed, we may assume that $L_{n}$ and $L_{n-1}$ intersect in the upper half plane.

The explicit values of $\mu_i$ given in the Proposition become thus 

\begin{center}$
\mu_i=\epsilon_1 \cdot m_{\min}\cdot \big(\sin\big(2\beta_i\big)+\frac{1}{b_i}\cdot \epsilon_2\big).
$\end{center}

We will use the fact that $\{1/b_i\ |\ i=1,...,n\}=\{b_i\ |\ i=1,...,n\}$.

We calculate some simple assertions. For $1\leq i \leq n$, we have $\pi/2n<\tan(\pi/2n)\leq |b_i| \leq\tan(\pi/2-\pi/2n)=1/\tan(\pi/2n)<2n/\pi$ and $\pi/2n<\sin(\pi/n)\leq |\sin(2\beta_i)|\leq 1$. This gives -- using the equality $\epsilon_2=n^{-6}$ --  the following relations
\begin{equation}\label{eqbi}
\begin{array}{rcl}
\pi/2n<&|b_i|&<2n/\pi;\\
\pi/2n<&|\sin(2\beta_i)|&\leq 1;\\
1/n<&|\sin(2\beta_i)+\frac{1}{b_i}\cdot \epsilon_2|&<2;\\
1/n^{11}\cdot m_{\min}<&|\mu_i|&<2m_{\min}/n^{10};
\end{array}
\end{equation}
and we see that $\mu_i$, $\sin(2\beta_i)$ and $b_i$ have the same sign.

For $1\leq i\leq n-2$, we obtain similarly the relation
\begin{equation}\label{eqai}
\begin{array}{rcl}
\pi/n<&|a_i|&<n/\pi.\end{array}
\end{equation}
We calculate now some coordinates of intersections of the lines of $\mathcal{B}$.

1. 
The $y$-coordinate of the intersection of $L_i$ and $L_j$ (for $i\not=j$) is equal to $(a_i-a_j)\cdot(\frac{1}{m_i}-\frac{1}{m_j})^{-1}$. Assuming that $\{i,j\}\not=\{n,n-1\}$, we have $|a_i-a_j|\geq \tan(\pi/n)-1/n>\pi/n-1/n>2/n$. Since $|\frac{1}{m_i}-\frac{1}{m_j}|\leq 2/{m_{\min}}$, we obtain the following assertion:
\begin{equation}\label{eqLiLj}
\begin{array}{c}
\mbox{\it The $y$-coordinate of } L_i\cap L_j, \mbox{\it  for }\{i,j\}\not=\{n,n-1\}\\
\mbox{\it  is (in absolute value) bigger than } m_{\min}/n.\end{array}\end{equation}
Note that the lines $L_{n-1}$ and $L_n$ may intersect at a very small $y$-coordinate.

\bigskip

2. 
The $y$-coordinate of the intersection of $L_i$ and $M_j$ is equal to $(a_i-b_j)\cdot (\frac{1}{m_i}-\frac{1}{\mu_j})^{-1}$. We calculate first -- using (\ref{eqbi}) and (\ref{eqai}) -- that $|a_i-b_j|\leq \max  |a_k|+\max |b_k|<n/\pi +2n/\pi=3n/\pi<n$. Secondly, we have $|\frac{1}{m_i}|\leq \frac{1}{m_{\min}}$ and -- using (\ref{eqbi}) -- obtain also $|\frac{1}{\mu_j}|>n^{10}/2m_{\min}$. We see that $|\frac{1}{m_i}-\frac{1}{\mu_j}|>n^{10}/3m_{\min}$.
\begin{equation}\label{eqLiMj}
\begin{array}{c}
\mbox{\it The $y$-coordinate of } L_i\cap M_j, \\
\mbox{\it  is (in absolute value) smaller than } 3m_{\min}/n^9.\\
\mbox{\it The $x$-coordinate is between  $a_i-3/n^9$ and $a_i+3/n^9$.}\end{array}\end{equation}

\bigskip

3. The $x$-coordinate of the intersection of $M_i$ and $M_j$ is equal to \[x_{ij}=\frac{\mu_i b_i-\mu_jb_j}{\mu_i-\mu_j}=\frac{\sin(2\beta_i)b_i-\sin(2\beta_j)b_j}{\sin(2\beta_i)-\sin(2\beta_j)+\epsilon_2\cdot({b_i}^{-1}-{b_j}^{-1})},\]
where $b_i=\tan(\beta_i)$, $b_j=\tan(\beta_j)$. We study now three cases:

3a) If $\beta_i+\beta_j=0$, the $x$-coordinate  $x_{ij}$ is equal to $0$, and the $y$-coordinate is negative. 

3b) Assume that $\beta_i+\beta_j=\pm \pi/2$, which implies that $\sin(2\beta_i)=\sin(2\beta_j)$ and $b_ib_j=1$, whence $1/b_i-1/b_j=b_j-b_i$. We find $x_{ij}=\sin(2\beta_i)\cdot (b_i-b_j)/ (\epsilon_2 \cdot (b_j-b_i))=-\sin(2\beta_i)/\epsilon_2$, which implies -- with (\ref{eqbi}) -- that $|x_{ij}|>(\pi/2n)/ n^{-6}>n^{5}$.

3c) Assume that $\beta_i+\beta_j \notin\{0,\pm \pi/2\}$. The trigonometric identities leads to 
$\sin(2\beta_i)\tan(\beta_i)-\sin(2\beta_j)\tan(\beta_j)=\tan(\beta_i+\beta_j)(\sin(2\beta_i)-\sin(2\beta_j))$, which implies that
\[x_{ij}-\tan(\beta_i+\beta_j)=-\tan(\beta_i+\beta_j)\cdot \frac{\epsilon_2\cdot (1/b_i-1/b_j)}{\sin(2\beta_i)-\sin(2\beta_j)+\epsilon_2\cdot (1/b_i-1/b_j)}.\]
We bound the values of this expression: $\pi/n<2\tan(\pi/2n)\leq |1/b_i-1/b_j| \leq 2/\tan(\pi/2n)<4n/\pi$, and $8/n^2=2/\pi^2\cdot(2\pi/n)^2<1-\cos(2\pi/n)=\sin(\pi/2)-\sin(\pi/2-2\pi/n)\leq |\sin(2\beta_i)-\sin(2\beta_j)|\leq 2$, and $\pi/n<\tan(\pi/n)\leq |\tan(\beta_i+\beta_j)|\leq \tan(\pi/2-\pi/n)=1/\tan(\pi/n)<n/\pi$. We obtain -- since $\epsilon_2=n^{-6}$ -- the following bounds

\begin{equation}\label{eqbb}
\begin{array}{rcl}
\pi/n<&|1/b_i-1/b_j|&<4n/\pi;\\
8/n^2<&|\sin(2\beta_i)-\sin(2\beta_j)|&\leq 2;\\
7/n^2<&|\sin(2\beta_i)-\sin(2\beta_j)+\epsilon_2(1/b_i-1/b_j)|&<3;\\
\pi/n<&|\tan(\beta_i+\beta_j)|&<n/\pi,
\end{array}
\end{equation}
and see that the expressions $\sin(2\beta_i)-\sin(2\beta_j)+\epsilon_2(1/b_i-1/b_j)$ and $\sin(2\beta_i)-\sin(2\beta_j)$ have the same sign.

The bounds (\ref{eqbb}) yield a minimal bound for $|x_{ij}-\tan(\beta_i+\beta_j)|$, which is $(\pi/n)\cdot \epsilon_2 \cdot (\pi/n)/3=\pi^2/3n^2\cdot \epsilon_2>3n^{-8}.$ Similarly, the maximal bound is 
$(n/\pi)\cdot \epsilon_2 \cdot (4n/\pi)/(7/n^2)=4/7\pi^2\cdot n^4\cdot \epsilon_2<n^{-2}$. We obtain the following relation
\begin{equation}\label{eqvaluedeltaij}
\begin{array}{rcl}
3/n^8&<|x_{ij}-\tan(\beta_i+\beta_j)|&<1/n^2.
\end{array}
\end{equation}
We study now the sign of $x_{ij}-\tan(\beta_i+\beta_j)$, which is the same as those of $-\tan(\beta_i+\beta_j)\cdot (1/b_i-1/b_j)\cdot \big(\sin(2\beta_i)-\sin(2\beta_j)\big)=-\big(\sin(2\beta_i)\tan(\beta_i)-\sin(2\beta_j)\tan(\beta_j)\big)\cdot \big(1/\tan(\beta_i)-1/\tan(\beta_j)\big)$.

Note that the function $x\mapsto \sin(2x)\tan(x)=2\sin(x)^2$ on $]-\pi/2;\pi/2[$ acts like $x\mapsto x^2$ (it is an even function, growing on $[0;\pi/2[$). We may thus replace $\sin(2x)\tan(x)$ by $x^2$ in the above expression without changing the sign. Similarly, we may replace $1/\tan(x)$ by $1/x$. The sign of $x_{ij}-\tan(\beta_i+\beta_j)$ is thus the same as the sign of $-({\beta_i}^2-{\beta_j}^2)/ (1/\beta_i-1/\beta_j)=(\beta_i+\beta_j)\cdot \beta_i\cdot \beta_j$.

\begin{equation}\label{eqsigndeltaij}
\begin{array}{c}
\mbox{\it The sign of } x_{ij}-\tan(\beta_i+\beta_j) \\
\mbox{\it  is the same as the sign of } (\beta_i+\beta_j)\cdot \beta_i\cdot \beta_j.
\end{array}\end{equation}

4. We describe now the order of the $x$-coordinates of the intersections of some line $M_i$ with the other lines of $\mathcal{B}$. Assume that $b_i=\tan(\beta_i)>0$ and for $D\in \mathcal{B}$, $D\not=M_i$, denote by $x_D$ the $x$-coordinate of the intersection of $M_i$ and $D$. Recall -- see (\ref{eqLiMj}) -- that $x_{L_j}\in ]a_j-3/n^9;a_j+3/n^9[$ for $j=1,...,n$. Furthermore, since $L_{n-1}$ and $L_{n}$ intersect on the upper half-plane and $a_{n-1}<0<a_n$, we see that $m_{n-1}>0>m_n$. Since $b_i>0$ and $\mu_i>0$, and since $|\mu_i|<|m_n|,|m_{n-1}|$ -- see (\ref{eqbi}) -- the intersection of $M_i$ with $L_{n-1}$ (respectively with $L_n$) has negative (respectively positive) $x$-coordinate. Thus, $x_{L_n} \in ]0;1/n+3/n^9[$ and $x_{L_{n-1}}\in ]-1/n-3/n^9;0[$. The positions of the $X_{L_j}$, for $j=1,...,n$ are given in Figure \ref{FigIntervalles1}.

\begin{figure}[ht]{\includegraphics[width=11cm]{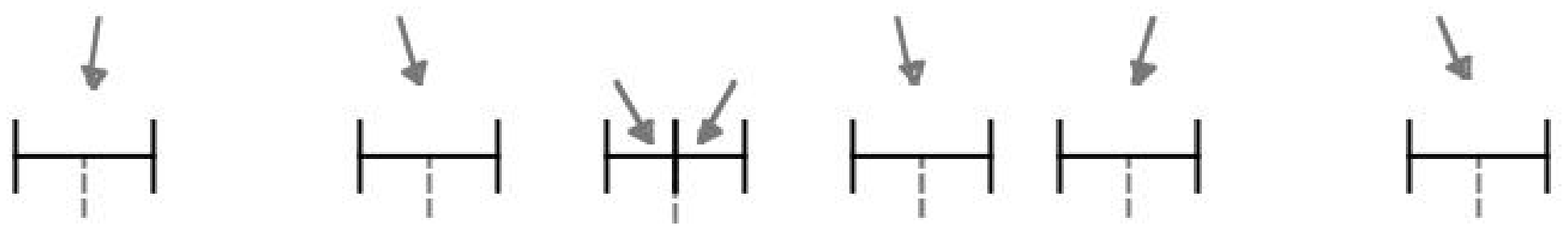}
\drawat{-70.5mm}{11mm}{$x_{L_{n-1}}$}%
\drawat{-59.5mm}{11mm}{$x_{L_n}$}%
\drawat{-115mm}{-3.2mm}{$\tan(-\pi/2+\pi/n)$}%
\drawat{-86mm}{-3.2mm}{$\tan(-\pi/n)$}%
\drawat{-54mm}{-3.2mm}{$\tan(\pi/n)$}%
\drawat{-38mm}{-3.2mm}{$\tan(2\pi/n)$}%
\drawat{-14mm}{-3.2mm}{$\tan(\pi/2-\pi/n)$}%
\drawat{-64.5mm}{-3.2mm}{$0$}%
\drawat{-22mm}{4mm}{$\cdots$}%
\drawat{-95mm}{4mm}{$\cdots$}%
\caption{The disposition of $x_{L_j}$ for $j=1,...,n$.\label{FigIntervalles1}}
}\end{figure}

We describe now the values of $x_{M_j}$, for $j\not=i$. Writing $z:=\beta_i+\beta_j$, the discussion made above -- in particular 3a), 3b), (\ref{eqvaluedeltaij}) and (\ref{eqsigndeltaij}) -- shows the following:
\[\begin{array}{llp{1.5cm}ll}
1]&\mbox{\it if } z=0 &\mbox{\it\ then} & x_{M_j}=0, x_{L_{n-1}}<x_{M_j}<x_{L_n};\\ 
 2]&\mbox{\it if }z=\pi/2 & \mbox{\it\ then} & x_{M_j}<-n^5;\\
 3]&\mbox{\it if } z<0 \mbox{\it\ or if\ } \pi/2<z<\pi& \mbox{\it\ then} &\tan(z)+3/n^8<x_{M_j}<\tan(z)+1/n^2;\\
 4]&\mbox{\it if } 0<z<\beta_i & \mbox{\it\ then} &\tan(z)-1/n^2<x_{M_j}<\tan(z)-3/n^8;\\
 5]&\mbox{\it if } \beta_i<z<\pi/2 & \mbox{\it\ then} &\tan(z)+3/n^8<x_{M_j}<\tan(z)+1/n^2.
\end{array}\]
We obtain the situation of Figure \ref{FigIntervalles2} (note that $x_{Y_0}=b_i$ and that there is no $x_{M_j}$ near $\tan(2\beta_i)$).
\begin{figure}[ht]{\includegraphics[width=11cm]{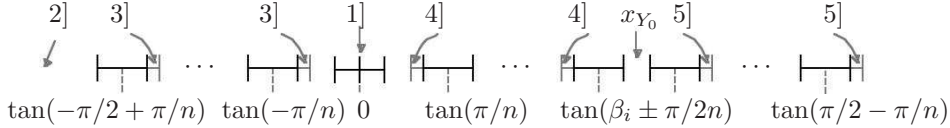}
\drawat{-70.5mm}{9.5mm}{$1]$}%
\drawat{-110mm}{9.5mm}{$2]$}%
\drawat{-101.8mm}{9.5mm}{$3]$}%
\drawat{-82mm}{9.5mm}{$3]$}%
\drawat{-60mm}{9.5mm}{$4]$}%
\drawat{-41mm}{9.5mm}{$4]$}%
\drawat{-27mm}{9.5mm}{$5]$}%
\drawat{-34mm}{9.5mm}{$x_{Y_0}$}%
\drawat{-7mm}{9.5mm}{$5]$}%
\drawat{-115.5mm}{-3.2mm}{$\tan(-\pi/2+\pi/n)$}%
\drawat{-87mm}{-3.2mm}{$\tan(-\pi/n)$}%
\drawat{-60mm}{-3.2mm}{$\tan(\pi/n)$}%
\drawat{-42mm}{-3.2mm}{$\tan(\beta_i\pm \pi/2n)$}%
\drawat{-14mm}{-3.2mm}{$\tan(\pi/2-\pi/n)$}%
\drawat{-69mm}{-3.2mm}{$0$}%
\drawat{-18mm}{2.8mm}{$\cdots$}%
\drawat{-92mm}{2.8mm}{$\cdots$}%
\drawat{-50mm}{2.8mm}{$\cdots$}%
\caption{The places of $x_{M_j}$, depending on the value of $z=\beta_i+\beta_j$.\label{FigIntervalles2}}
}\end{figure}

In particular,  every element  of $U=\{x_{M_j} | j=1,...,n, j\not=i \} \cup \{x_{Y_0}\}$ is between two consecutive $x_{L_j}$'s. Furthermore, between two ${x_{L_j}}$'s there is exactly one element of $U$, except for one place (near $\tan(2\beta_i)$), where there is no element of $U$.

Doing the same for every $M_i$ (the situation for $\beta_i<0$ is similar, as the construction of the $M_j$'s is symmetric), we see that between two consecuting intervals containing one $a_j$ (of the form $]a_j-3/n^9;a_j+3/n^9[$ or $]0;1/n+3/n^9[$ or $]-1/n-3/n^9;0[$ ), exactly $n/2$ pairs of the lines $M_1,...,M_n, Y_0$ intersect.  Furthermore, these intersect only one triangle of $\mathcal{A}$, which is the triangle touching $Y_0$ (see \ref{eqLiLj} and \ref{eqLiMj}).
We obtain the situation of Figure \ref{FigEi}.

\begin{figure}[ht]
{\includegraphics[height=3cm]{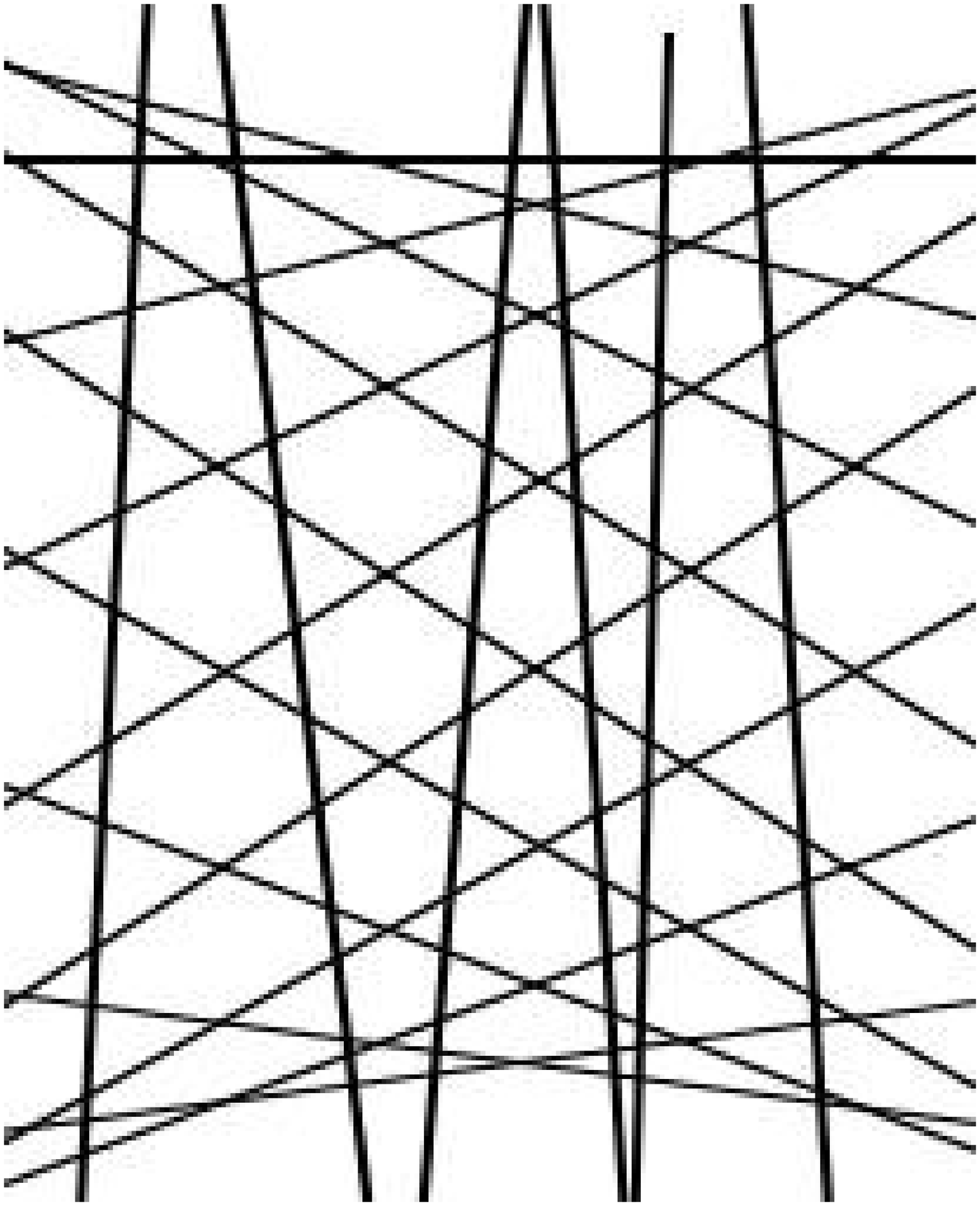}}\drawat{1mm}{25mm}{$Y_0$}%
\hspace{2 cm} {\includegraphics[height=3cm]{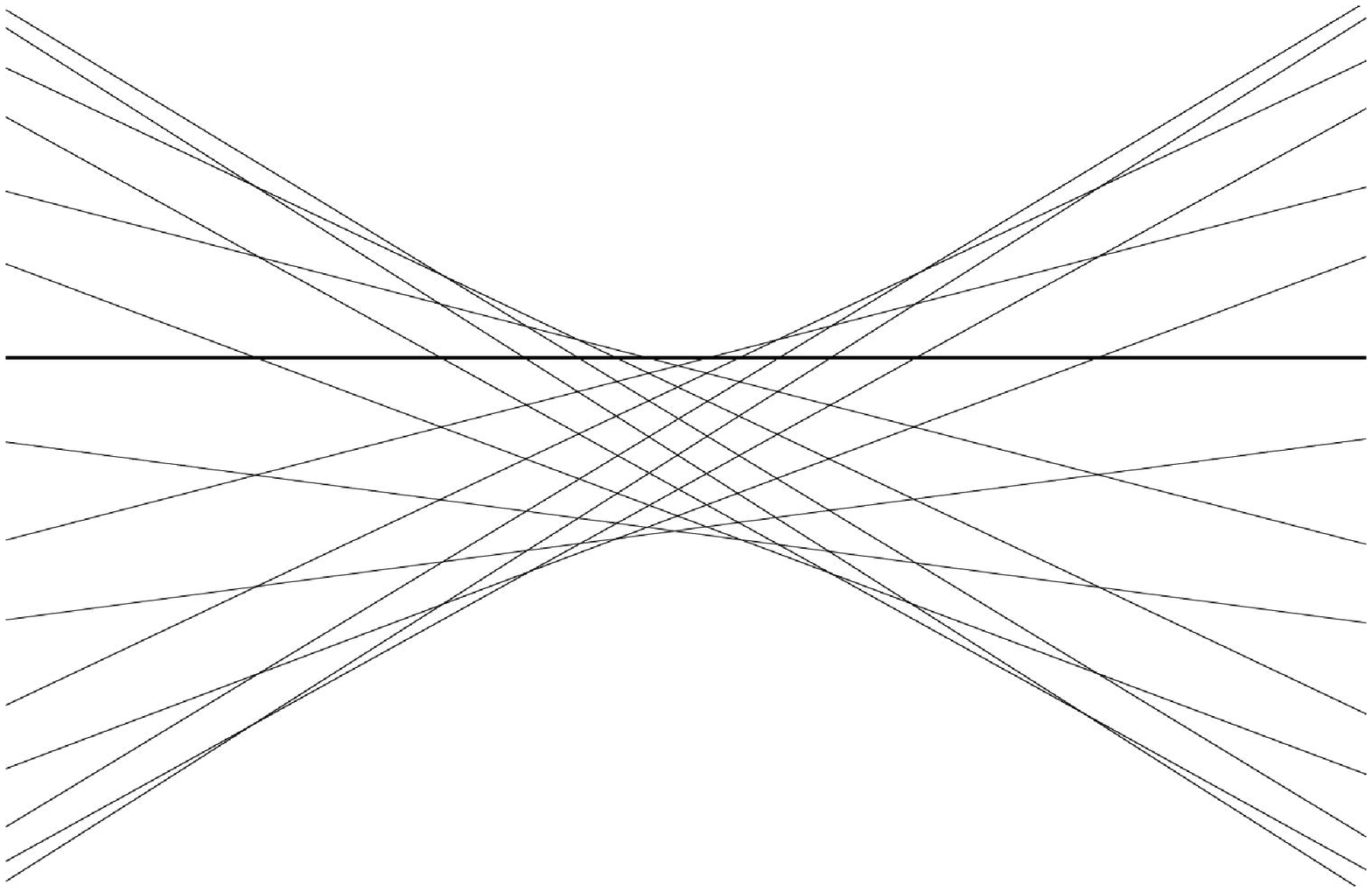}}\drawat{0mm}{18mm}{$Y_0$}
 \caption{The configurations of the lines $M_{1},M_{2},\dots,M_{n}$ and $Y_0$ -- with at the left also some of the lines $L_{1},...,L_{n}$.\label{FigEi}}\end{figure}

Each of the $n-1$ triangles of $\mathcal{A}$ that touches $Y_0$ is replaced in $\mathcal{B}$ by $n+1$ triangles, and we obtain also $n$ triangles at the extremities (formed by the lines $M_i$ and $M_j$ with $\beta_i+\beta_j=\pm \pi/2$). The arrangement $\mathcal{B}$ has thus exactly $n^2$ triangles more than $\mathcal{A}$, and every segment of $Y_0\subset \mathcal{B}$ is used in one triangle. This achieves to prove the Proposition.
\end{proof}

\begin{corollary}\label{Coro:PseudoSeq}
Let $n>1$ be an odd integer, write $m:=2n-1$ and assume that $\overline{a_3^s}(n)> n(n-3)/3$. 

We have $\overline{a_3^s}(m)\geq \overline{a_3^s}(n)+(n-1)^2$, i.e. $m(m-2)/3-\overline{a_3^s}(m)\leq n(n-2)/3-\overline{a_3^s}(n)$.
In particular, if $\overline{a_3^s}(n)=\lfloor n(n-2)/3 \rfloor$ then $\overline{a_3^s}(m)=\lfloor m(m-2)/3 \rfloor$.
\end{corollary}
\begin{proof}
Let $\mathcal{A}$ be an affine arrangement of $n$ pseudo-lines with more than $n(n-3)/3$ triangles. There exists one pseudo-lines $Y_0\in \mathcal{A}$ that touches $n-2$ triangles of the arrangement; we stretch this to the line $y=0$, arrange the intersections of $Y_0$ with the other pseudo-lines $L_1,...,L_{n-1}$ to satisfy the conditions of Proposition \ref{Prp:duplicationG}, and stretch every pseudo-line  $L_i\in \mathcal{A}\backslash Y_0$ so that the segments of $L_i$ that touch $Y_0$ become segments of lines. Then, adding the lines $M_1,...,M_n$ of Proposition \ref{Prp:duplicationG} to our arrangement gives an arrangement of $2n-1$ pseudo-lines with $(n-1)^2$ triangles more than $\mathcal{A}$.
\end{proof}

We are now able to prove Theorem \ref{TheoNewSequence}, with the help of the following new maximal arrangement:
\begin{figure}[ht]{\includegraphics[width=11.6 cm]{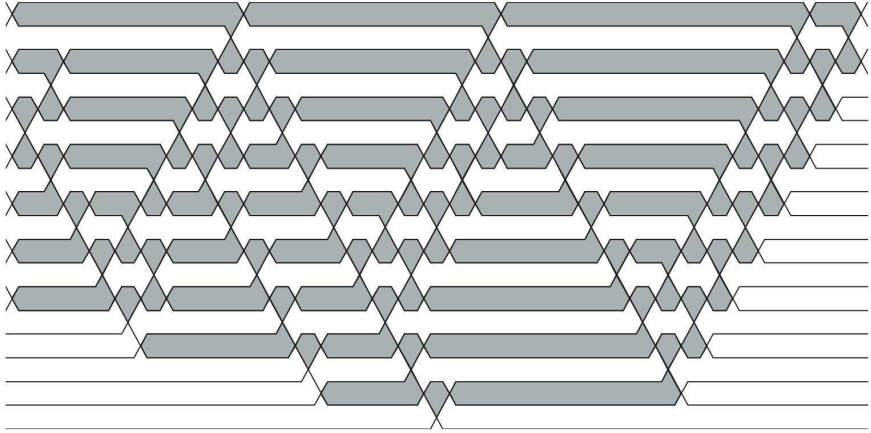}} \caption{A maximal affine arrangement of $19$ pseudo-lines with $107$ triangles  \label{Fig19}}\end{figure}

\begin{proof}[Proof of Theorem \ref{TheoNewSequence}]

1. There exists a maximal affine arrangement of $n=15$ lines with $n(n-2)/3$ triangles, found by Simmons in 1972 (see \cite{bib:Sim}). We give a new one, having the properties needed to apply Proposition \ref{Prp:duplicationG}.
For any $\epsilon >0$, let $\mathcal{A}_{\epsilon}$ be the arrangement $\{L_1,...,L_{15}\}$ of $15$ lines given by 
\begin{center} $L_i:=\{(x,y) \in \Rn \ | \ y=m_i (x-a_i)\}$,\end{center}
where

\begin{center}$\begin{array}{lllp{1 cm}lll}
a_1=\tan(-6\pi/14)& & m_1=1.66 & & a_{9~}=\tan(2\pi/14)& & m_9=-12.4\\
a_2=\tan(-5\pi/14)& & m_2=4.4& & a_{10}=\tan(3\pi/14)& & m_{10}=-22\\
a_3=\tan(-4\pi/14)& & m_3=3.28& & a_{11}=\tan(4\pi/14)& & m_{11}=-4.8\\
a_4=\tan(-3\pi/14)& & m_4=14.4& & a_{12}=\tan(5\pi/14)& & m_{12}=-5.3\\
a_5=\tan(-2\pi/14)& & m_5=13.1& & a_{13}=\tan(6\pi/14)& & m_{13}=-1.86\\
a_6=\tan(-\pi/14)& & m_6=-65& & a_{14}=-\epsilon& & m_{14}=50\\
a_7=0& & m_7=0& & a_{15}=\epsilon & & m_{15}=-45\\
a_8=\tan(\pi/14)& & m_8=-52
\end{array}$
\end{center}

\begin{figure}[ht]{\includegraphics[width=9 cm]{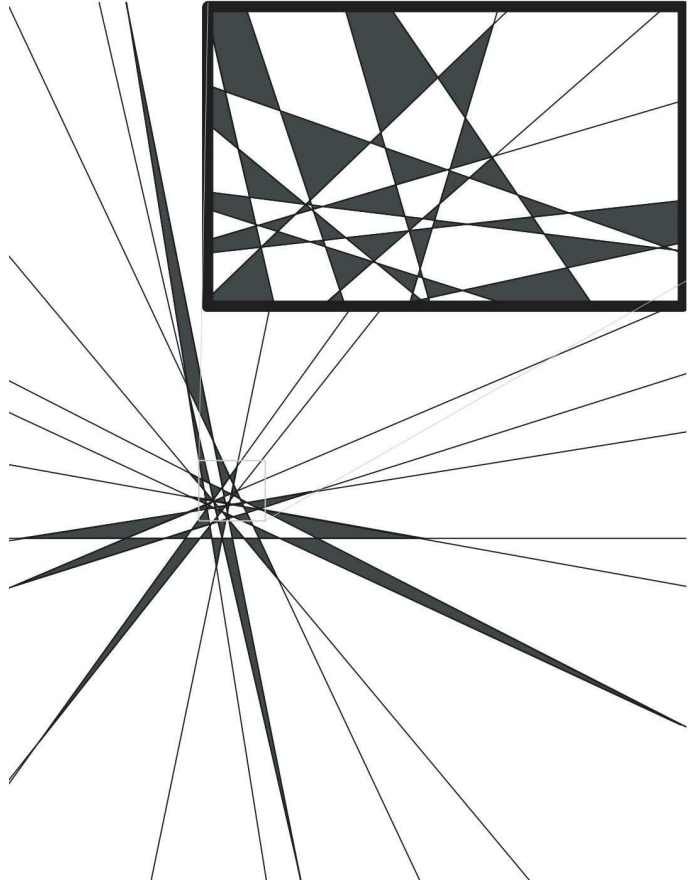}} \caption{The perfect arrangement  $\mathcal{A}_{\epsilon}$ of $15$ lines, beginning of the induction.\label{Fig15}}\end{figure}

We can verify by inspection that the configuration is perfect (see Figure \ref{Fig15}), if $\epsilon$ is small enough. 
Furthermore, we can apply Proposition \ref{Prp:duplicationG} to get a perfect arrangement of $29$ lines. By iterating,  starting from $\mathcal{A}_{\epsilon}$, for a small $\epsilon$ (which depends on $t$), one obtains an  arrangement of $14 \cdot 2^t +1$ lines, for any integer $t\geq1$.

2. For any $\epsilon >0$, let $\mathcal{A}_{\epsilon}$ be the arrangement $\{L_1,...,L_{7}\}$ of $7$ lines given by 
\begin{center} $L_i:=\{(x,y) \in \Rn \ | \ y=m_i (x-a_i)\}$,\end{center}
where

\begin{center}$\begin{array}{lllp{1 cm}lll}
a_1=\tan(-2\pi/6)& & m_1=3 & & a_{5}=\tan(2\pi/6)& & m_5=-3\\
a_2=\tan(-\pi/6)& & m_2=1& & a_{6}=-\epsilon & & m_{6}=-7\\
a_3=\tan0& & m_3=0& & a_{7}=\epsilon & & m_{7}=7\\
a_4=\tan(\pi/6)& & m_4=-1\\
\end{array}$
\end{center}

\begin{figure}[ht]{\includegraphics[width=6 cm]{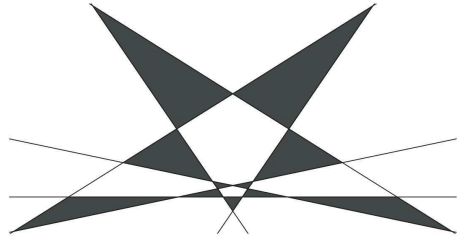}} \caption{The maximal arrangement  $\mathcal{A}_{\epsilon}$ of $7$ lines, beginning of the induction.\label{Fig7}}\end{figure}
We see that the arrangement has $11$ triangles, and use Proposition \ref{Prp:duplicationG}  to get an maximal arrangement of $13$ lines. By iterating, one gets -- for any integer $t\geq1$ -- a maximal arrangements of $n=6 \cdot 2^t +1$ lines, with $\lfloor n(n-2)/3\rfloor$ triangles.

3. The arrangement of Figure \ref{Fig19} is a maximal arrangement of $19$ pseudo-lines with $107$ triangles. Iterating Corollary \ref{Coro:PseudoSeq} we find -- for any integer $t\geq1$ -- a maximal arrangements of $n=18 \cdot 2^t +1$ pseudo-lines, with $\lfloor n(n-2)/3\rfloor$ triangles.
\end{proof}

\section{Description of the computer algorithm}\label{Sec:DescAlgo}
In this Section, we discuss a computer algorithm to search for affine pseudo-line arrangements with many triangles. The problem of finding line arrangements with many triangles is a
geometrical one. It is possible to formulate a related combinatorial
problem for pseudo-line arrangements. We will work with \emph{wiring diagrams} (introduced by Goodman \cite{bib:Goo}), see Figures \ref{Fig19} and \ref{FigXXX}. In this representation the $n$ curves are  $x$-monotone and are restricted to n $y$-co$y$-coordinates except for some local switches where adjacent lines cross.

The information of an affine arrangement of $n$ pseudo-lines is stored into  a $(n-1)\times m $ matrix $M$, where $m$
 is some positive integer. Each column contains some $X$'s and describes the crossings at some $x$-coordinate; an "$X$" at the height $i$ means that the pseudo-lines $i$ and $i+1$ intersect there. We typically add suggestive horizontal lines to these matrices to obtain
pseudo-line diagrams as seen in Figure \ref{FigXXX}.
\begin{figure}[ht]
{\includegraphics[width=5cm]{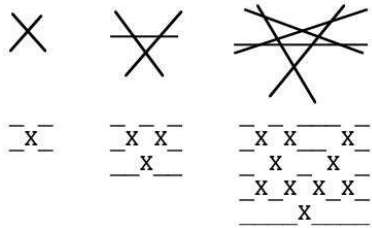}} 
\caption{Pseudo-line Diagrams describing the line configurations above
\label{FigXXX}}
\end{figure}

The polygons of pseudo-line affine arrangements represented by the matrices are easy to compute and the
notation lends itself to several {}``pruning'' ideas.
From now on, we will write $M=(M_{1},...,M_{m})$ and refer to this
matrix as the pseudo-line diagram. The search algorithm is:

\begin{flushleft}
{\tt
Function~depth\_first\_search($M$)

Denote~$M=(M_{1},...,M_{k})$.

If~$M$~is~a~pseudo-line~affine~arrangement

~~~~Count~its~triangles

else

~~~~Generate~the~list~$L$~of~all~possible~choices~of~$M_{k+1}$.

~~~~For~each~$M_{k+1}\in L$,

~~~~~~~~depth\_first\_search($(M_{1},...,M_{k},M_{k+1})$)

~~~~End~for

End~if}
\end{flushleft}
We add some "pruning criteria" to reduce the search:
\begin{enumerate}
\item
In any given column, no two crosses may be adjacent.
\item
It is not permitted to put a cross between two pseudo-lines that have
already crossed.
\item
Without loss of generality, we may impose
that all crosses be placed as far to the left as possible.
\end{enumerate}
A vector $(M_{1},...,M_{k})$ with insufficiently many intersections
but otherwise satisfying the three above properties is called a \emph{partial} or \emph{incomplete} affine arrangement. Although an arrangement is incomplete we are able to compute its triangles and to see that some segments are already \emph{unused} (i.e. not touching a triangle of the future complete arrangement). Since we are looking for diagrams with many triangles, we must have
few unused edges -- this allows us to discard some partial arrangements without compromising the search.
\begin{enumerate}
\item
If on column $k$ we put a cross in row $j$ that closes a triangle,
then the polygons in column $k$ and rows $j-1$ and $j+1$ cannot
be triangles.
\item
If the budget of unused segments is exhausted, we will have some forced dispositions of the crosses, to ensure that every remaining segment will touch one triangle.
\end{enumerate}
\section{Computer Results}
We have looked for maximal affine arrangements of $n$ pseudo-lines. Perfect arrangements are only possible when $n\equiv3\textrm{ or }5\pmod{6}$,
 and we achieved such for $n=3,5,9,15,17,21,23,27,29$. Because several
of our heuristics exploit the low number of unused edges, as the
unused edge budget increases, the search quickly becomes intractable when looking for imperfect arrangements.

\begin{proposition}
\label{Prp:computr}
The maximum number of affine triangles in a pseudo-line arrangement found by our algorithm are given in the following tables:

\begin{center}$\begin{array}{|c|c|c|c|c|c|c|c|c|c|c|c|c|c|c|c|c|c|c|c|c|c|c|c|c|c|c|c|}
\hline
n & 3 & 4 & 5 & 6 & 7 & 8 & 9 & 10 & 11 & 12 & 13  \\
\hline
\overline{a_3^s}(n)&1 & 2 & 5 & 7 & 11 &  14 & 21 & 25 &  32 & 37 & 47  \\
\hline
\end{array}$

\vspace{0.1 cm}

$\begin{array}{|c|c|c|c|c|c|c|c|c|c|c|c|c|c|c|c|c|c|c|c|c|c|c|c|c|c|c|c|}
\hline
n  & 14 & 15 & 17 & 19 & 21 & 23 & 27  & 29 \\
\hline
\overline{a_3^s}(n)  & 53 & 65 & 85&107 &133 &  {161} & { 225} & 261  \\
\hline
\end{array}$
\end{center}
\end{proposition}

\end{document}